\documentclass{amsart}
\usepackage{amsthm}
\usepackage{amsfonts}
\usepackage{amsmath}
\usepackage{amsmath, amssymb, amsfonts, amsthm, amscd}
\usepackage{caption}
\usepackage{graphicx}
\usepackage{comment}
\usepackage{amscd}
\usepackage{tikz-cd}
\usepackage{enumitem} 
\usepackage{pb-diagram}
\usepackage{tikz}
\usepackage[hidelinks]{hyperref}
\usetikzlibrary{matrix}
\usetikzlibrary{calc}
\usepackage{bbm}
\usepackage{dutchcal}
\usepackage{secdot}

\usepackage{wasysym}
\usepackage{titlesec}
\usepackage{enumitem}
\newcommand{\KK}{\mathbb{K}}

\newcommand{\Aut}{\mathrm{Aut}}

\theoremstyle{plain}
\newtheorem{theorem}{Theorem}[section]
\newtheorem{propos}[theorem]{Proposition}
\newtheorem{lem}[theorem]{Lemma}
\newtheorem{cor}[theorem]{Corollary}
\theoremstyle{definition}
\newtheorem{const}[theorem]{Construction}
\newtheorem{de}[theorem]{Definition}
\newtheorem{ex}[theorem]{Example}
\theoremstyle{remark}
\newtheorem{zam}[theorem]{Remark}

\begin{document}

\title[Varieties with...a finite number of automorphism group orbits]{Varieties with a torus action of complexity one having a finite number of automorphism group orbits}

\author{Sergey Gaifullin and Dmitriy Chunaev}

\thanks{The research was done within the framework of the HSE Fundamental Research Program in 2023.}



\address{HSE University, Faculty of Computer Science, Pokrovsky Boulvard 11, Moscow, 109028 Russia. \ \ \ \ \ \ 
\linebreak 
Lomonosov Moscow State University, Leninskie gory, 1, Moscow, 119991, Russia.\ \ \ \ \ \ \ \ \ \ \ \ \ \ \ \ \ \ \ \ \ \ \ \ \ 
\linebreak
Moscow Center of Fundamental and Applied Mathematics, Moscow, Russia.}
\address{HSE University, Faculty of Computer Science, Pokrovsky Boulvard 11, Moscow, 109028 Russia. \ \ \ \ \ \ 
\linebreak 
Lomonosov Moscow State University, Leninskie gory, 1, Moscow, 119991, Russia.}

\maketitle

\begin{abstract}
In this work we obtain sufficient conditions for a variety with a torus action of complexity one to have a finite number of automorphism group orbits.
\end{abstract}

\section{Introduction}

Let $\mathbb{K}$ be an algebraically closed field of characteristic zero. Let us concider an algebraic variety $X$ over $\mathbb{K}$. We denote by $\mathrm{Aut}(X)$ the group of regular automorphisms of the variety~$X$. The variety $X$ is decomposed onto a union of $\mathrm{Aut}(X)$-orbits with respect to the natural $\mathrm{Aut}(X)$-action on $X$. For some natural classes of varieties numbers of $\mathrm{Aut}(X)$-orbits are finite. Examples of classes of varieties with finite numbers of $\mathrm{Aut}(X)$-orbits are, for example, the following ones
\begin{itemize}  
\item algebraic groups;
\item homogeneous spaces of algebraic groups;
\item toric varieties; 
\item spherical varieties (this class strictly contains the class of toric varieties), see~\cite{K} and~\cite{T};
\item flexible normal affine surfaces, the definition of a flexible variety one can find in~\cite{AFKKZ}. In the same work it is proved that the automorphism group of a flexible variety action on the smooth locus transitively.
\item homogeneous varieties, i.e. varieties with unique $\mathrm{Aut}(X)$-orbits. Examples of homogeneous varieties that are not homogeneous spaces of algebrais groups  are given in~\cite{AZ}. Also smooth affine varieties with actions of reductive groups with an open orbits are homogeneous varieties, see the proof in~\cite{GSh}.
\end{itemize}

The class of toric varieties is the class of varieties admitting an action of an algebraic torus of complexity zero, i.e. with an open orbit.  In this paper we consider the class of affine irreducible algebraic varieties with a torus action of complexity 1. That are varieties admitting effective actions of a torus with dimension equal to dimension of the variety minus one. We prove that a variety  $X$ with an action of complexity 1, satisfying some additional conditions, namely rational without nonconstant invertible functions and with a finitely generated class group, having a finite number of  $\mathrm{Aut}(X)$-orbits if it admits a homogeneous locally nilpotent derivation of the horizontal type. Recall that a locally nilpotent derivation is of the horizontal type if it acts nontrivially on rational invariants of the torus.  Locally nilpotent derivations correspond to algebraic subgroups of $\mathrm{Aut}(X)$ isomorphic to the additive group of the ground field. Such subgroups we call $\mathbb{G}_a$-subgroups. The above conditions for locally nilpotent differentiation in terms of this subgroup are that this $\mathbb{G}_a$-subgroup must be normalized by the action of the torus and it acts nontrivially on the rational invariants of this torus.

A particular cases of varieties with a torus action of complexity 1 are trinomial hypersurfaces given by equations
$$
T_{01}^{l_{01}}\ldots T_{0n_0}^{l_{0n_0}}+T_{11}^{l_{11}}\ldots T_{1n_1}^{l_{1n_1}}+T_{21}^{l_{21}}\ldots T_{2n_2}^{l_{2n_2}}=0, \qquad n_0\geq 0,\  n_1,n_2>0.
$$
If $n_0=0$, then we assume the first monomial to be equal to 1. Orbits of automorphism group of nonrigid (i.e. admitting a niontrivial $\mathbb{G}_a$-subgroup in $\Aut(X)$) trinomial  hypersurfaces are investigated in~\cite{G-S}. In particular, in~\cite{G-S} it is proved that the number of orbits is finite. Note that if a trinomial hypersurface is rigid, then the number of  $\mathrm{Aut}(X)$-orbits is infinite. The automorphism group of rigid trinomial hypersurfaces are described in~\cite{AG2}.  In the present paper we generalize the result of~\cite{G-S} proving a sufficient condition for the number of $\Aut(X)$-orbits on a trinomial variety to be finite, see Theorem~\ref{mmaaiinn} and Corollary~\ref{gortip}. Recall that a trinomial variety is an affine variety given by a coordinated system of trinomial equations, see the rigorous definition in Construction~\ref{odyn}.
Each trinomial variety $X$ is a product of a trinomial variety $Y$ and an $m$-dimensional affine space by definition. If  $m=0$, then $Y=X$.
A sufficient condition for the number of  $\Aut(X)$-orbits to be finite is that  the variety $Y$ is not rigid, i.e. admittes a nontrivial $\mathbb{G}_a$-subgroup of $\mathrm{Aut}(Y)$. 
This condition can be reformulated explicitely in terms of degrees of variables in equations, i.e. in terms of data we use to construct the variety. Trinomial varieties are studied  in~\cite{HS, HH, AHHL, H-W, ABHW, H-W2, G, E-G-S} and others. Trinomial varieties are particular cases of a varieties with a torus actions of complexity 1. A criterion of rigidity for trinomial varieties is obtained in~\cite{E-G-S}, see also works \cite{Ar} and \cite{G}, where some particular results were obtained. This criterion plays a key role in the proof of the sufficient condition of finiteness of number of $\Aut(X)$-orbits.

An arbitrary rational irreducible affine variety without non-constant invertible functions with a finitely generated class group admitting a  torus action of complexity one can be reduced to the case of a trinomial variety using the Cox construction, see for example~\cite{ADHL}. Recall that the Cox construction matches the so-called Cox ring to any normal variety $X$ without nonconstant invertible functions and with a finitely generated  divisor class group. If this ring is finitely generated, one can consider its spectrum $\overline{X}$, which is called the total coordinate space of $X$. The variety $X$ can be canonically realized as a categorical quotient of $\overline{X}$ by the action of a quasitorus (that is, the product of an algebraic torus and a finite abelian group). It is proved in~\cite{H-W}, that the total coordinate ring of a (rational irreducible affine variety without non-constant invertible functions) variety admitting a  torus action of complexity one is a trinomial variety. Using the standard technique for automorphisms of the Cox ring, a locally nilpotent differentiation of horizontal type homogeneous with respect to the character group of the maximal torus can be lifted to one on the total coordinate space. Therefore,  there will be only a finite number of orbits of the automorphism group on the total coordinate space. The connection between automorphism groups of a variety and its total coordinate space is described in~\cite{AG1}.  Using the results of this work, we show that there are only a finite number of orbits of the automorphism group on the original variety. Thus, we obtain that a sufficient condition for the finiteness of the number of $\Aut(X)$-orbits for an arbitrary (rational irreducible affine without non-constant invertible functions) variety with a  torus action  of complexity 1 consists in the existence of a locally nilpotent derivation of horizontal type homogeneous with respect to the character group of the maximal torus.

The authors are grateful to I.V. Arzhantsev for fruitful discussions. The first author is a Young Russian Mathematics award winner and would like to thank its sponsors and jury.

\section{Locally nilpotent derivations}
\label{os}
In this section we gather preliminaries on the field of locally nilpotent derivations. Detailed information in this area can be found, for example, in the book~\cite{F}. 

Let $A$ be a commutative assosiative domain over the field $\mathbb{K}$. We assume that the algebra~$A$ is finitely generated, i.e. it is the algebra $\mathbb{K}[X]$ of regular functions on an affine variety~$X$.

\begin{de}
   By  a {\it derivation} of the algebra $A$ we mean a linear operator $\delta\colon A\rightarrow A$, sutisfying the Leibnits rule:
    $\delta(ab)=a\delta(b)+b\delta(a)$.
\end{de}
\begin{de}
A derivation $\delta\colon A\rightarrow A$ is called {\it locally nilpotent} (LND), if for each $a\in A$ there exists a positive integer~$n$ such, that $\delta^n(a)=0$.
\end{de}

Suppose we have a grading on $A$ by a commutative group~$G$
$$
A=\bigoplus_{g\in G}A_g.
$$
A derivation $\delta$ is called homogeneous, if it takes homogeneous elements to homogeneous ones. It is easy to show that for each homogeneous LND $\delta$ there is such $g_0\in G$, that $\delta$ maps $A_g$ to $A_{g+g_0}$. The element $g_0$ is called the {\it degree} of the derivation $\delta$. It is not difficult to prove that every derivation can be decomposed onto a finite sum of homogeneous ones, which are called {\it homogeneous components} of the derivation. 

Suppose we have a $\mathbb{Z}$-grading. Let $\delta$ be an LND of $A$. Then $\delta=\sum\limits_{i=l}^k\delta_i$. In~\cite{R} it is proved that the extremal components $\delta_l$ and $\delta_k$ are also locally nilpotent. This implies that a $\mathbb{Z}$-graded algebra admits a nonzero LND if and only if it admits a nonzero $\mathbb{Z}$-homogeneous LND. This statement can be easily extended to the case of a $\mathbb{Z}^n$-grading.

Locally nilpotent differentiations are closely related to automorphisms of the algebra (the variety). Namely, we can consider the {\it exponent } of an LND $\delta$
$$
\exp(\delta)=\mathrm{id}+\delta+\frac{\delta^2}{2!}+\frac{\delta^3}{3!}+\ldots
$$
Since $\delta$ is locally nilpotent, when we apply $\exp(\delta)$ to an element $a\in A$, the sum is finite. The exponent of an LND is an automorphism. Moreover each LND $\delta$ corresponds to the following subgroup of automorphisms 
$$
\mathcal{H}_\delta=\{\exp(s\delta)\mid s\in\KK\}.
$$
This is a $\mathbb{G}_a$-subgroup, i.e. an algebraic subgroup in automorphism group (the automorphism group itself usually is not algebraic) isomorphic to the additive group of the field $\KK$.

Gradings on $A=\KK[X]$ by the free abelian group $\mathbb{Z}^n$ are in a natural bijection with algebraic actions of the algebraic torus $T=(\mathbb{K}^\times)^n$ on $X$. (Here $\mathbb{Z}^n$ is identified with the group of characters $\mathfrak{X}(T)$ of the torus $T$.) An LND $\delta$ is $\mathbb{Z}^n$-homogeneous if and only if the torus $T$ is contained in the normalizer of the corresponding $\mathbb{G}_a$-subgroup $\mathcal{H}_\delta$. We often say that  a derivation is $T$-homogeneous instead of  $\mathfrak{X}(T)$-homogeneous.

Each derivation of $\KK[X]$ can be uniquely extended to a derivation of the field of rational functions $\KK(X)$.
\begin{de}
Suppose we are given by an action of a torus  $T$ on a variety $X$. We say that a $T$-homogeneous LND $\delta$ has  the {\it vertical type}, if its extension on  $\KK(X)$ vanishes on the field of $T$-invariant fonctions $\KK(X)^T$. If $\delta$ does not vanish on $\KK(X)^T$, then $\delta$ has the {\it horizontal type}.
\end{de} 

\begin{de}
A variety $X$ is called {\it rigid}, if the algebra $\KK[X]$ does not admit any nonzero LNDs.
\end{de}

\section{Trinomial varieties}

In this section we investigate orbits of automorphism group of a trinomial variety. First of all let us give a rigorous definition of a trinomial variety according to~\cite{H-W}.

\begin{const}\label{odyn} \cite[Construction 1.1]{H-W}. Suppose we are given by positive integers~$r$ and~$n$, a nonnegative integer  $m$ and $q \in \left\{ {0, 1} \right\}$. Let us fix a partition  $n = n_{q} +  \ldots+ n_{r}$ of~$n$ onto positive integer summands. 
Let us consider the polynomial algebra $B$ in $m+n$ variables. These variables we denote $T_{ij}$ and~$S_k$:
$$
B= \KK \left[ {T_{ij}, S_{k} \,|\, q \leq i \leq r, \!\ 1 \leq j \leq n_{i}, \!\ 1 \leq k \leq m} \right].
$$
For each $i = q, \ldots, r$ we fix a tuple of positive integers $l_{i} = \left( {l_{i1},  \ldots , l_{in_{i}}} \right)$. So we can consider the following monomial:
$$
	T_{i}^{l_{i}} = T_{i1}^{l_{i1}}  \ldots T_{in_{i}}^{l_{in_{i}}} \in B. 
$$
Now we define the {\it trinomial algebra} $R(A)$, which we construct by some data $A$.  These data are different for two types of trinomial algebras. 
\par \emph{Type 1.} $ q = 1, \!\ A = (a_{1},  \ldots, a_{r}), \!\ a_{j} \in \KK, \!\ a_{i} \neq a_{j} $ if $i \neq j$.  Let us put  $I = \left\{ {1, \ldots, r - 1} \right\}$ and 
\begin{equation*}
	g_{i} = T_{i}^{l_{i}} - T_{i+1}^{l_{i+1}} - (a_{i+1} - a_{i}) \in B, \ i \in I.
\end{equation*}
\par \emph{Type 2.} $ q = 0$, 
\begin{equation*}
	A = \begin{pmatrix}
		a_{10} \ a_{11} \ a_{12} \cdots a_{1r} \\
		a_{20} \ a_{21} \ a_{22} \cdots a_{2r} \\
	\end{pmatrix}
\end{equation*}
is such a matrix with elements from $\mathbb{K}$, that every two columns are linearly independent. Let us put $I = \left\{ {0,  \ldots , r - 2} \right\}$ and 
\begin{equation*}
	g_{i} = \det \begin{pmatrix}
		T_{i}^{l_{i}} \ T_{i+1}^{l_{i+1}} \ T_{i+2}^{l_{i+2}} \\
		a_{1i} \ a_{1i+1} \ a_{1i+2}\\
		a_{2i} \ a_{2i+1} \ a_{2i+2}
	\end{pmatrix} \in B, \ i \in I.
\end{equation*}
For both types $R(A) = B/(g_{i}\, |\, i \in I)$.
\end{const}
\begin{de} The variety $X(A) = \mathrm{Spec}(R(A))$ is called a {\it trinomial variety}.  The type of a trinomial variety is the type of the corresponding trinomial algebra.\end{de} 

\begin{zam}
It is easy to see that the dimension of $X(A)$ equals $m+n-r+1$.
\end{zam}

\begin{zam} By definition, each trinomial variety $X(A)$ is isomorphic to a product of another trinomial variety $Y(A)$ and an $m$-dimensional affine space $X(A)\cong Y(A)\times \mathbb{A}^m$. To obtain the algebra of regular functions on  $Y(A)$ one should eliminate generators $S_k$  from $R(A)$. The type of $Y(A)$ coincide with the type of $X(A)$.
\end{zam}

We can define an action of an algebraic torus $\mathbb{T}\cong (\KK^\times)^{n+m-r}$ of complexity 1 on the variety $X(A)$. It is known that the character group of an algebraic torus is a free abelian group, and setting of an action of a torus of dimension $k$ on $X$ is equivalent to setting of a $\mathbb{Z}^k$-grading on the algebra of regular function $\KK[X]$. It is more convenient to define the $\mathbb{T}$-action on $X(A)$ in these terms, i.e. to define $\mathbb{Z}^{m+n-r}$-grading on $R(A)$. We define this grading as follows let us put that the variable $T_{ij}$ has the degree $w_{ij}$ and the variable $S_k$ has degree $v_k$. These degrees with relations of commutativity and $\deg T_i^{l_i}=\deg T_j^{l_j}$ generate the group $\mathbb{Z}^{m+n-r}$, see  more details  in~\cite{H-W}. We need the folowing one-dimensional subtori in the torus $\mathbb{T}$. The subtorus $\Lambda_{s,u,v}\cong \KK^\times$, where $1\leq u,v\leq n_s$ acts by the following rule:
$$
t\cdot T_{su}=t^{l_{sv}}T_{su}, \qquad t\cdot T_{sv}=t^{-l_{su}}T_{sv},
$$
$$
t\cdot T_{pq}=T_{pq}\text{ for all } (p,q),\text{ except }(s,u)\text{ and }(s,v).
$$
In case of trinomial varieties of type 2 we define one more one-dimensional subtorus  $\Omega$ in $\mathbb{T}$. An element $t\in \Omega$ multiply $T_{i1}$ by $t^{\prod_{j\neq i}l_{j1}}$, and $T_{is}$ for $s\geq 2$ do not change.

In~\cite[Theorem~4]{E-G-S} the following criterion of rigidity of a trinomial variety is obtained.
\begin{propos}\label{EGSh}
A trinomial variety $X(A)$ is not rigid if and only if one of the following conditions holds:
\begin{enumerate}[leftmargin=2\parindent]
\item[1)] $m \neq 0$,
\item[2)] The variety $X(A)$ has type 1 and there exists such an $a \in \lbrace1,\ldots,r\rbrace$, that for all $i \in \lbrace1,\ldots,r\rbrace\setminus\lbrace a\rbrace$ there is $j(i) \in\lbrace1,\ldots,n_i\rbrace$, such that $l_{ij(i)}=1$,
\item[3)] The variety $X(A)$ has type 2 and ane of the following conditions holds:

\begin{enumerate}[leftmargin=2\parindent]
	\item[a)] There are two such numbers $a, b \in\lbrace0,\ldots,r\rbrace$, that for each $i \in\lbrace0,\ldots,r\rbrace\setminus\lbrace a,b\rbrace$ there is $j(i) \in \lbrace1,\ldots,n_i\rbrace$, such that $l_{ij(i)}=1$,
	\item[b)]  There are three such numbers $a, b, c \in\lbrace 0,\ldots,r\rbrace$, that for each  
 $$k \in\lbrace0,\ldots,r\rbrace\setminus\lbrace a,b,c\rbrace$$ there is $j(k) \in \lbrace1,\ldots,n_k\rbrace$, such that $l_{kj(k)}=1$, and for each $i \in \lbrace a,b\rbrace$ there is $v(i) \in\lbrace1,\ldots,n_i\rbrace$, such that $l_{iv(i)}=2$ and for all $w \in\lbrace1,\ldots,n_i\rbrace$ the numbers $l_{iw}$ are even.
\end{enumerate}
	
\end{enumerate}
	
\end{propos}

\begin{zam}
For the variety $Y(A)$ the first condition does not hold. Therefore the variety $Y(A)$ is not rigid if and only if one of the conditions 2 and 3 holds.
\end{zam}

The first main result of this paper is the following sufficient condition for a trinomial variety to have finite number of orbits of automorphism group. 

\begin{theorem}\label{mmaaiinn}
If the variety $Y(A)$ is not rigid, then  the variety $X(A)$ has finite number of $\mathrm{Aut}(X(A))$-orbits. 
\end{theorem}
\begin{proof}
Suppose $Y(A)$ is not rigid. Let us fix a subset $J$ of the set of variables $T_{ij}$. This subset can be chosen by $2^n$ ways. Let us prove that each set  
$$
L(J)=\{x\in X(A)\mid T_{ij}(x)=0\Leftrightarrow T_{ij}\in J\}
$$
is contained in a finite union of $\mathrm{Aut}(X)$-orbits. This implies that the number of orbits is finite.  
We firstly consider the case $J\neq\varnothing $. 
Denote by $L_Y(J)$ the analogical subset, but in the variety $Y$:
$$
L_Y(J)=\{y\in Y(A)\mid T_{ij}(y)=0\Leftrightarrow T_{ij}\in J\}.
$$
Let us prove the following lemma. 
\begin{lem}
Suppose $J\neq \varnothing$.  Then for every two points $\alpha$ and $\beta$ in $L_Y(J)$ we can take $\alpha$ to $\beta$, by an action of an element of $\mathbb{T}$. 
\end{lem}\begin{proof}
Let $T_{is}\in J$. Acting on $\alpha$ by one-dimensional tori $\Lambda_{i,s,v}$, we can make the values of variables $T_{iv}$, which are not in $J$, equal to any nonzero numbers. In particular we can make values of variables $T_{iv}$ in the image of $\alpha$ to be equal to their values in $\beta$. 

Further proof we carry out separately for different types of $Y(A)$. Let $Y(A)$ be of the type 1. We can consider linear combinations of equations that give $Y(A)$, to obtain $T_i^{l_i}-T_j^{l_j}=a_j-a_i$. Since $T_i^{l_i}(\alpha)=0$, we have $T_j^{l_j}(\alpha)=a_i-a_j\neq 0$. Analogically $T_j^{l_j}(\beta)=a_i-a_j$. Hence all the variables $T_{ju}$ have nonzero values at $\alpha$ and $\beta$. It is easy to see, that using subtori $\Lambda_{j,u,v}$ we can take $\alpha$ to $\alpha'$ such that values of all the variables $T_{ju}$ at $\alpha'$ and $\beta$ coincide. Thus by consistent applying elements of  $\mathbb{T}$ we can take $\alpha$ to a point with the same values of all the coordinates as at $\beta$, i.e. to~$\beta$.

Now let $Y(A)$ is of type 2. Again we can consider linear combinations of equations that give $Y(A)$, to obtain  the following equations
\begin{equation*}
	\det \begin{pmatrix}
		T_{p}^{l_{p}} \ T_{q}^{l_{q}} \ T_{u}^{l_{u}} \\
		a_{1p} \ a_{1q} \ a_{1u}\\
		a_{2p} \ a_{2q} \ a_{2u}
	\end{pmatrix} =0
\end{equation*}
for all triples $(p,q,u)$. Substituting $u=i$ and the points $\alpha$ and $\beta$ we obtain, that $\lambda T_{p}^{l_p}(\alpha)+\mu T_{q}^{l_q}(\alpha)=0$ for some nonzero $\lambda$ and $\mu$ and $\lambda T_{p}^{l_p}(\beta)+\mu T_{q}^{l_q}(\beta)=0$. If for all $j\neq i$ it holds $T_j^{l_j}(\alpha)=0$, then for all $j$ we have $T_j^{l_j}(\alpha)=T_j^{l_j}(\beta)$. Otherwise let us chose such $j\neq i$ that $T_j^{l_j}(\alpha)\neq 0$. By an element of $\Omega$ we can take $\alpha$ to $\alpha'$ such, that $T_j^{l_j}(\alpha')=T_j^{l_j}(\beta)$. But then the obtained equations imply that for each $k$ it holds $T_k^{l_k}(\alpha')=T_k^{l_k}(\beta)$. Using  $\Lambda_{k,u,v}$, we can take $\alpha'$ to $\beta$ analogically to the case of type 1.

The lemma is proved.
\end{proof}

The above lemma implies that if $J$ is nonempty, then for every two points $\alpha$ and $\beta$ in $L(J)$ acting by an element $t$ of $\mathbb{T}$ we can move $\alpha$ to $\alpha'=t\cdot\alpha$ such that values of all coordinates  $T_{ij}$ coincide at $\alpha'$ and~$\beta$. From the other hand the group of parallel translations acts on the affine space $\mathbb{A}^m$ transitively. We can take  $\alpha'$ to $\beta$ by the element of this group.
\begin{zam}
If $J$ is nonempty we do not use the condition that $Y(A)$ is not rigid.
\end{zam}
It remains to prove that $L(\varnothing)$ is contained in finite union of $\mathrm{Aut}(X(A))$-orbits. Since the variety $Y(A)$ is not rigid, there exists such an LND $\delta$ on $R(A)$, that $\delta(T_{ij})\neq 0$ for some pair $(i,j)$. By~\cite[Principle~5]{F} $\delta(T_{ij})$ is not divisible by $T_{ij}$ in $R(A)$. Therefore there exists such a point  $x\in L(\{T_{ij}\})$, that $\delta(T_{ij})(x)\neq 0$. Let us consider the group $G$, generated by $\mathbb{T}$ and the $\mathbb{G}_a$-subgroup $\{\exp(s\delta)\mid s\in\KK\}$. The subgroup $G$ of $\mathrm{Aut}(X(A))$ is algebraically generated, i.e. it is generated by algebraic groups. By~\cite[Proposition~1.3]{AFKKZ} each $G$-orbit is locally closed in Zariski topology. In particular, the dimension of this orbit is well defined. Since $Gx$ contains $L(\{T_{ij}\})$, but it is not contained in  $L(\{T_{ij}\})$, the dimension of $Gx$ is greater than $\dim\,L(\{T_{ij}\})=\dim\,\mathbb{T}=\dim\,X(A)-1$. Therefore, $Gx$ is open in $X(A)$. The complement of this orbit is a may be reducible variety of dimension not grater than $\dim\,\mathbb{T}$. But the dimension of $\mathbb{T}$-orbit of any point in $L(\varnothing)$ equals $\dim\,\mathbb{T}$. Hence, the dimension of $G$-orbit of any point in $L(\varnothing)$ is not less than $\dim\,\mathbb{T}$. Therefore, $L(\varnothing)$ is covered by a finite number of $G$-orbits. (This orbits can be not contained in $L(\varnothing)$.) Since $G$-orbits are contained in $\mathrm{Aut}(X(A))$-orbits, this implies the statement of the theorem.

\end{proof}

\begin{zam}
Using LNDs from~\cite[Section~4]{G} and \cite[Lemma~10]{G} one can explicitly construct LNDs on nonrigid trinomial varieties. It can be explicitly proved that in conditions of Theorem~\ref{mmaaiinn} the set $L(\varnothing)$ is contained in a unique $\mathrm{Aut}(X(A))$-orbit.
\end{zam}

\begin{zam}\label{zz2}
   It follows from the proof of Theorem~\ref{mmaaiinn} that each subset $L(J)$, where $J\neq\varnothing$, is a $\mathbb{T}$-orbit, and the subset $L(\varnothing)$ is covered by a finite number of orbits of the group generated by  $\mathbb{T}$ and any $\mathbb{G}_a$-subgroup, nontrivially acting on $Y(A)$.
\end{zam}

The fact that the $\mathbb{T}$-normalizable subgroup $\mathbb{G}_a$-subgroup acts nontrivially on $Y(A)$ can be formulated in terms of the corresponding LNDs. 
\begin{lem}\label{eqt}
A homogeneous with respect to $\mathbb{T}$ LND $\delta$ on $R(A)$ has vertical type if and only in $\delta(T_{ij})=0$ for all pairs~$(i,j)$.
\end{lem}
\begin{proof}
 It is easy to prove that the field of $\mathbb{T}$-invariants is contained in the field generated by all monomials $T_i^{l_i}$. Exactly, in case of type 1 these two fields coincide and in case of type 2 the field of $\mathbb{T}$-invariants is generated by quotients of such monomials. In any case if $\delta(T_{ij})=0$ for all pairs $(i,j)$, then the derivation has the vertical type.  
 
  Conversely if $\delta$ does not equal to zero at some $T_{ij}$, then by Remark~\ref{zz2} the group $G$, generated by $\mathbb{T}$ and~$\mathcal{H}_\delta$, acts on $X(A)$ with a finite number of orbits. Hence, one of these orbits is open. Therefore, the field of $G$-invariants coincide with $\mathbb{K}$.  On the other hand 
    $$\KK(X(A))^G=\KK(X(A))^{\mathbb{T}}\cap \mathrm{Ker}\,\delta.$$
    At the same time $\KK(X(A))^{\mathbb{T}}\neq\KK$. Therefore, $\delta$ has the horizontal type.
\end{proof}

We obtain the following assertion.

\begin{cor}\label{gortip}
If $R(A)$ admits a $\mathbb{T}$-homogeneous LND $\delta$ of horizontal 
type, then the number of  $\mathrm{Aut}(X(A))$-orbits on $X(A)$ is 
finite. Moreover, in this case there are only finite number of $G$-orbits on $X(A)$, where $G$ is the subgroup of $\mathrm{Aut}(X(A))$, generated
by $\mathbb{T}$ and the $\mathbb{G}_a$-subgroup, corresponding to $\delta$.
\end{cor}
\begin{proof}
The fact that if there exists a $\mathbb{T}$-homogeneous LND of the horizontal type, then the number of $\mathrm{Aut}(X(A))$-orbits is finite, follows from Lemma~\ref{eqt} and Theorem~\ref{mmaaiinn}. Finiteness of  the number of $G$-orbits in these conditions follows from Remark~\ref{zz2}.
\end{proof}

\section{Varieties with a torus action of complexity one}

It is proved in~\cite[Corollary~1.9]{H-W}, that any normal rational irreducible affine variety $Z$ without nonconstant invertible functions, admiting a torus action of complexity one, can be canonically obtained by a categorical quotient of a trinomial variety $\overline{Z}$ by an action of a quasitorus $H\subseteq \mathbb{H}$, where $\mathbb{H}$ is the the centralizer of  $\mathbb{T}$ in $\mathrm{Aut}(X(A))$ which is also quasitorus.
This realization of the variety $Z$ by the categorical quotient is the Cox realization of this variety, see~\cite{ADHL}. It is proved in~\cite{AG1} that there is the following exact sequence of groups: 
$$
1\rightarrow H\rightarrow \widetilde{\mathrm{Aut}(\overline{Z})}\rightarrow \mathrm{Aut}(Z)\rightarrow 1,
$$
where $\widetilde{\mathrm{Aut}(\overline{Z})}$ is the normalizer of $H$ in automorphism group of $\overline{Z}$. Here the mapping $\widetilde{\mathrm{Aut}(\overline{Z})}\rightarrow \mathrm{Aut}(Z)$ is the restriction of automorphisms from $\mathbb{K}[\overline{Z}]$ to $H$-invariants. 

\begin{lem}\label{l41}
 If the number of $\widetilde{\mathrm{Aut}(\overline{Z})}$-orbits on $\overline{Z}$ is finite, then the number of $\mathrm{Aut}(Z)$-orbits on $Z$ is finite.
\end{lem}
\begin{proof}
The quotient morphism
$$
\pi\colon \overline{Z}\rightarrow \overline{Z}/\!/H=Z
$$
is surjective and the image of each $\widetilde{\mathrm{Aut}(\overline{Z})}$-orbit coincide with a $\mathrm{Aut}(Z)$-orbit.    
\end{proof}

\begin{propos}\label{p42}
Suppose 
$$Z=X(A)/\!/H$$
is the Cox  realization of an irreducible affine algebraic variety $Z$ with a torus action of complexity 1. Then if there exists an LND of $R(A)$, that is $H$-homogeneous of degree zero and at least one variable $T_{ij}\in R(A)$ is not in its kernel, then the number of $\mathrm{Aut}(Z)$-orbits is finite.  
\end{propos}
\begin{proof}
 Let us consider the $\mathbb{G}_a$-subgroup corresponding to the LND from the statement of the proposition. Denote by $G$ the subgroup generated by $\mathbb{T}$ and this $\mathbb{G}_a$-subgroup. Since $T_{ij}$ is not in the kernel of the LND, Remark~\ref{zz2}  implies that the number of $G$-orbits on $X(A)$ is finite. On the other hand since this LND is $H$-homogeneous of degree zero, the corresponding $\mathbb{G}_a$-subgroup lies in $\widetilde{\mathrm{Aut}(X(A))}$, i.e. normilizes  $H$ and even commutes with $H$. Therefore, $G\subseteq \widetilde{\mathrm{Aut}(X(A))}$. By Lemma~\ref{l41}, the number of $\mathrm{Aut}(Z)$-orbits is finite.
\end{proof}

Proposition~\ref{p42} gives us a sufficient condition for the number of $\mathrm{Aut}(Z)$-orbits to be finite in terms of Cox realization of this variety. However, it would be useful to have such a condition in the internal terms of the variety $Z$. It is given by the following theorem.
\begin{theorem}\label{glavt}
Let $Z$ be a normal rational irreducible affine variety with only constant invertible functions and with an action of a torus $\widehat{T}$ of complexity~1. Suppose $Z$, admits a $\widehat{T}$-homogeneous LND of the horizontal type. Then the number of $\mathrm{Aut}(Z)$-orbits is finite.
\end{theorem}
\begin{proof}
By our assumpotion $Z$, admits a $\widehat{T}$-homogeneous LND of the horizontal type. It correspondes to a $\mathbb{G}_a$-subgroup $\Omega$ in $\mathrm{Aut}(Z)$, that acts nontrivially on $\mathbb{K}(Z)^{\widehat{T}}$. By \cite[Theorem~4.2.3.2]{ADHL} there exists a $\mathbb{G}_a$-subgroup $\overline{\Omega}$ in $\widetilde{\mathrm{Aut}(\overline{Z})}$ such, that its restriction to $\KK[Z]$ coinside with $\Omega$. Then $\overline{\Omega}$ is a $\mathbb{T}$-homogeneous $\mathbb{G}_a$-subgroup acting on $\KK(\overline{Z})^{\mathbb{T}}=\mathbb{K}(Z)^{\widehat{T}}$  nontrivially. So $\overline{\Omega}$ corresponds to an LND of horizontal type on~$\overline{Z}$. By Corollary~\ref{gortip}, the number of $G$-orbits on $\overline{Z}$ is finite, where $G$ is the group generated by $\mathbb{T}$ and~$\overline{\Omega}$. Since $G\subseteq \widetilde{\mathrm{Aut}(Z)}$, the number of $\widetilde{\mathrm{Aut}(Z)}$-orbits on $\overline{Z}$ is finite. By Lemma~\ref{l41}, the number of $\mathrm{Aut}(Z)$-orbits on $Z$ is finite.
\end{proof}

\begin{ex}
   Let us concider the hypersurface
    $$X=\{x_1\ldots x_k(y_1^{b_1}\ldots y_m^{b_m}+z_1^{c_1}\ldots z_l^{c_l})=uv_1^{r_1}\ldots v_p^{r_p}\}\subseteq \KK^{k+l+m+p+1}.$$
   One can define an action of $k+l+m+p-1$-dimensional torus on this hypersurface analogicaly to the action on a trinomial hypersurface. More precisely,  each variable corresponds to a degree and we have two linear relationship on these degrees. We can consider the following LND of the horizontal type:   
    $$
    \delta(u)=b_1x_1\ldots x_ky_1^{b_1-1}\ldots y_m^{b_m},\qquad \delta(y_1)=v_1^{r_1}\ldots v_p^{r_p},
    $$
    on the other variables $\delta$ equals zero. It is easy to see that $\delta$ is indeed locally nilpotent. It is of horizontal type since it is nonzero at the following invariant of the torus    
    $$
    \frac{uv_1^{r_1}\ldots v_p^{r_p}}{x_1\ldots x_kz_1^{c_1}\ldots z_l^{c_l}}
    $$
    The variety $X$ is rational since the open subset $\{v_1^{r_1}\ldots v_p^{r_p}\neq 0\}\subseteq X$ is isomorphic to the open subset in affine space. $X$ does not admit nonconstant invertible functions since it does not admit nonconstant homogeneous invertible functions. This is true since there are no nonzero $\mathfrak{X}(T)$-degrees opposite to each other. Finally, the normality of the variety $X$ follows from Serre's normality criterion, since the set of singular points in $X$ has codimension 2.
    By Theorem~\ref{glavt}, the number of $\mathrm{Aut}(X)$-orbits on $X$ is finite.
\end{ex}

\end{document}